\documentclass{article}

\usepackage{parskip}

\usepackage[margin=1.6in]{geometry}

\usepackage{amssymb,latexsym,amsmath,amsthm,graphicx}

\numberwithin{equation}{section}
\newtheorem{thm}{Theorem}[section]
\newtheorem{lemma}[thm]{Lemma}

\newtheorem{q}{Question}

\newtheorem*{remark}{Remark}

\newcommand{\bep}{\begin{prob}}
\newcommand{\ep}{\end{prob}}
\newcommand{\be}{\begin{equation}}
\newcommand{\ee}{\end{equation}}

\usepackage[colorinlistoftodos]{todonotes}

\newcommand{\CC}{\mathbb{C}}

\newcommand{\ol}{\overline}

\begin{document}
\author{Dmitry Khavinson and Erik Lundberg}

\title{ A note on arclength null quadrature domains }

\maketitle

{\centering\footnotesize \emph{This paper is dedicated to the fond memory of Harold S. Shapiro.} \par \bigskip}

\begin{abstract}
    We prove the existence of a roof function for  arclength null
    quadrature domains having finitely many boundary components.
    This bridges a gap toward classification
    of arclength null quadrature domains
    by removing an a priori assumption from previous classification results.
\end{abstract}

\section{ Introduction }




A domain $\Omega \subset \CC$ is referred to as an
\emph{arclength null-quadrature domain} (arclength NQD) if the identity
\begin{equation}\label{eq:ALNQD}
\int_{\partial \Omega} g(z) ds(z) = 0,
\end{equation}
is satisfied for all functions $g$
in the Smirnov space $E^1(\Omega)$
(a class of analytic functions suitable for integration along $\partial \Omega$,
see below),
where $ds(z)$ denotes the arclength element.

In order for the test class $E^1(\Omega)$ to be non-trivial, one should impose local rectifiability of the boundary $\partial \Omega$.  In this paper, we will further impose a smoothness assumption on $\partial \Omega$, see Theorem \ref{thm:main}.  We note that some regularity such as smoothness is needed if one wants to avoid the exotic pathological non-Smirnov examples of arclength NQD constructed in \cite[Sec. 2]{KLT}.

\smallskip

Arclength NQDs are related to a free boundary problem for the Laplace equation (a correspondence that we will strengthen in this note).
Additional motivation for studying arclength NQDs comes from fluid dynamics \cite{EL}, \cite{MT2} and minimal surfaces \cite{MT1}.
Arclength NQDs appear as blow up solutions 
in a class of free boundary problems \cite{JeKa}, \cite{JeKa2016}, \cite{Henrik}.

The following problem was stated in \cite{KLT} and restated with discussion in \cite{EL}, \cite{Lopen}.

\noindent {\bf Problem:} Classify arclength NQDs. 

\smallskip

The related problem of classifying 
\emph{area null-quadrature domains} (area NQDs), where the integration is over $
\Omega$ with respect to area measure, was completely solved
in 1981 by M. Sakai \cite{Sakai} who showed that
area NQDs fall into one of the following four cases:

\begin{itemize}
 \item the exterior of an ellipse
 \item the (non-convex) exterior of a parabola 
 \item a halfplane
 \item a domain whose boundary is a proper subset of a line
\end{itemize}

The halfplane and the exterior of a disk are NQDs for both area and arclength.  The other examples constructed in \cite{EL} show that the class of arclength NQDs is quite rich and includes multiply-connected examples with boundary curves parameterized by elliptic functions.

Area and arclength NQDs have natural higher-dimensional analogs of
volume and surface area NQDs (respectively) using appropriate test classes of harmonic functions.
The classification of such volume NQDs is an ongoing investigation, see \cite{FrSa}, \cite{Karp}, \cite{Sha}, and the references therein.
Classification of surface area NQDs is an interesting uncharted territory.

While the classification of planar area NQDs is completely resolved by Sakai's results, the classification problem for arclength NQDs remains open, and progress has been stifled by a nagging question, stated below, concerning the existence of a so-called roof function
for arclength NQDs.

\subsection{Domains that admit a roof function}\label{sec:q}


A sufficient condition for a domain $\Omega$ to be an arclength NQD is that $\Omega$ admits a \emph{roof function}, a positive function $u$ harmonic in $\Omega$ such that the gradient $\nabla u$ coincides with the inward-pointing unit normal
vector along $\partial \Omega$.
Note that this boundary condition is stronger than a mere Neumann condition 
since it is imposed on the \emph{gradient} (not just the normal derivative),
and it implies that $u$ itself is constant along each
component of $\partial \Omega$ (with possibly distinct constants on 
different boundary components).

Domains that admit roof functions are called \emph{quasi-exceptional domains},
and when the roof function is further assumed to have constant Dirichlet data (not just piecewise constant), 
they are referred to as \emph{exceptional domains}.
It was shown in \cite{EL} that quasi-exceptional domains are arclength NQDs, i.e., we have
\begin{thm}\label{thm:ELforward}
If $\Omega$ admits a roof function
then $\Omega$ is an arclength NQD.
\end{thm}

Let us sketch the proof of this result to provide some context for what follows.
Suppose $u$ is a roof function for $\Omega$.
Then notice that the analytic completion $f$
of $u$ has a single-valued derivative.
Namely, $f'$ is just the complex conjugate of $\nabla u$.
Thus, $i f'(z)$ provides an analytic continuation 
of the conjugate of the unit tangent vector throughout all of $\Omega$.
Using the relation between the arclength element $ds(z)$ and the unit tangent vector, we have
$$ ds(z) = i f'(z) dz .$$
This allows us to restate the condition \eqref{eq:ALNQD} as requiring, for all $g \in E^1(\Omega)$,
$$\int_{\partial \Omega} g(z) f'(z) dz = 0.$$
That the above integral vanishes is a consequence of Cauchy's theorem,
and we conclude that $\Omega$ is an  arclength NQD.
However, in order to make this argument rigorous, one must verify that
$g \cdot f'$ is in the test class $E^1(\Omega)$, defined below.
In fact, one can show $f' \in H^\infty(\Omega)$ (we recall the definition of $H^\infty$ below)
using potential theoretic estimates 
on $u$ (this step relies on the positivity of $u$, see \cite{EL}), and this implies $g \cdot f' \in E^1(\Omega)$. 

Substantial progress has been made classifying exceptional domains and quasi-exceptional domains
\cite{HHP, KLT, MT1, MT2, EL}.
These results rely on the existence of a roof function, 
raising the following question that was posed in \cite{EL} (cf. \cite{Lopen})
asking whether the converse to Theorem \ref{thm:ELforward} holds.

\begin{q}\label{q:main}
Does every arclength NQD admit a roof function?
\end{q}

Under the assumption that $\Omega$ has finitely many boundary components,
we give an affirmative answer to this question in the next section (see Theorem \ref{thm:main}), thus showing that the above-mentioned classification results for quasi-exceptional domains represent definitive progress on the classification problem for arclength NQDs.

The proof of Theorem \ref{thm:main}
has two key ingredients.
The first is a characterization of the Smirnov space of analytic functions, a result of Havinson and Tumarkin
stated as Theorem \ref{thm:HaTuSm} below,
that allows reversing the main step in the proof of Theorem \ref{thm:ELforward} by establishing that the vanishing of integrals stated in the arclength NQD condition \eqref{eq:ALNQD} guarantees the analytic continuation of the tangent vector to all of $\Omega$.
This leads to a candidate roof function, but showing positivity requires a second key idea (in this instance potential theoretic) from \cite[Thm. II]{Kjellberg} which is based on the proof of the Denjoy-Carleman-Ahlfors theorem.
In order to highlight the utility of this method, we point out an interesting comparison with \cite{EL}:
the assumed positivity of the roof function
was used along with a growth estimate in \cite{EL} in order to apply \cite[Thm. II]{Kjellberg}, whereas
here we will use an adaptation of \cite[proof of Thm. II]{Kjellberg} in order to show the positivity
of a candidate roof function, see Lemma \ref{lemma:pos} below.


\noindent {\bf Acknowledgments.}
We are grateful to the anonymous referees whose careful reading of the paper helped clarify, and in some places correct, the exposition.

\section{Existence of a roof function for arclength NQDs}\label{sec:roof}





First, we recall the definitions of the
Hardy spaces $H^p(D)$ and the Smirnov spaces $E^p(D)$.
A function $g$ analytic in $D$ is said to belong to $E^p(D)$ if there exists a sequence of cycles $\gamma_k$
homologous to zero, rectifiable, and converging to the boundary $\partial D$ 
(in the sense that $\gamma_k$ eventually surrounds each compact sub-domain of $D$),
such that:
$$\sup_{\gamma_k} \int_{\gamma_k} |g(z)|^p |dz| < \infty.$$
  
On the other hand, a function $g$ analytic in $D$ is said to belong to $H^p(D)$ if the function $|f|^p$ admits a harmonic majorant in $D$. 
Basic properties of these spaces can be found in \cite{Duren}, \cite{Fisher}, \cite{HaTu1960}.

We recall a key result
from the theory of Smirnov spaces (see \cite[Ch. 10]{Duren} for a more detailed overview).  The following result
due to Havinson and Tumarkin \cite{HaTu1960}
provides an extension (to the multiply-connected setting) 
of a result of Smirnov \cite[Thm. 10.4]{Duren}.

\begin{thm}[Havinson, Tumarkin]\label{thm:HaTuSm}
Let $D$ be a finitely-connected domain with rectifiable boundary.
Suppose $g \in L^1(\partial D)$,
and the function $h$ defined by
$$ h(w) := \int_{\partial D} \frac{ g(\zeta)}{\zeta - w} d\zeta$$ 
vanishes for all 
$ w \in \CC \setminus \overline{D}$.
Then $h \in E^1(D)$ and
has boundary values $g$
almost everywhere on $\partial D$.
\end{thm}

We now state our result addressing Question \ref{q:main}.

\begin{thm}\label{thm:main}
Suppose $\Omega$ is an arclength NQD and that the boundary $\partial \Omega$ consists of finitely many smooth curves.
Then $\Omega$ admits a roof function.
\end{thm}

\begin{remark}
Note that $\Omega$ is necessarily unbounded,
since otherwise the constant functions are in the test class $E^1(\Omega)$ and fail to satisfy the null quadrature condition.
Also, $\Omega$ may have boundary components that are unbounded.
\end{remark}

In the proof of Theorem \ref{thm:main}, the following important lemma will be used to construct a candidate roof function. 

\begin{lemma}\label{lemma:tangvec}
Let $\Omega$ be as in Theorem \ref{thm:main}. The unit tangent vector $T(z)$ of $\partial \Omega$ admits an extension throughout $\Omega$ to an analytic function $h(z) \in H^\infty(\Omega) \cap C(\ol{\Omega})$.
\end{lemma}

\begin{proof}[Proof of Lemma \ref{lemma:tangvec}]

From the arclength null quadrature condition we have, for an arbitrary function 
$g \in E^1(\Omega)$,
$$\int_{\partial \Omega} g(z) ds = 0, $$
where $ds$ denotes the arclength element.
Write
$$\int_{\partial \Omega} g(z) ds = \int_{\partial \Omega} g(z) \overline{T(z)} dz ,$$
where $T(z)$ denotes the unit tangent vector to $\partial \Omega$.

Let $\phi:K \rightarrow \Omega$
be a conformal mapping from a bounded circular domain $K$ to $\Omega$.
Recall that a circular domain is a finitely-connected domain whose boundary components are all circles,
and also recall that each finitely-connected domain is conformally equivalent to a circular domain \cite[Ch. 3]{Goluzin}.
Then we have for each $g \in E^1(\Omega)$
$$\int_{\partial K} g(\phi(w)) \phi'(w) \overline{T(\phi(w))} dw = 0.$$

By Havinson and Tumarkin's extension \cite{HaTuDef} 
(to the multiply-connected setting) of a result of Keldysh and Lavrentiev \cite{KeLa}
we have that $g \in E^1(\Omega)$
is equivalent to $g(\phi(w)) \phi'(w) \in H^1(K) = E^1(K)$,
and in particular the functions $g(\phi(w)) \phi'(w)$ generate all of $E^1(K)$.
Hence, 
$$\int_{\partial K} G(w) \overline{T(\phi(w))} dw = 0, \quad \text{for all } G \in E^1(K).$$

This implies
that the function $\kappa$ defined by
$$\kappa(w) := \int_{\partial K} \frac{ \overline{T(\phi(\xi))}}{\xi - w} d\xi$$
vanishes for all 
$w \in \CC \setminus \overline{K}$.
By Theorem \ref{thm:HaTuSm} we have
that $\kappa \in E^1(K)$ and
has boundary values $\overline{T(\phi(w))}$
almost everywhere on $\partial K$.
Since the boundary components of $K$ are real-analytic (they are circles), we have $E^1(K) = H^1(K)$ \cite[p. 182]{Duren}.
Since $H^1 \subset N^+$, the Smirnov class, we conclude \cite[Thm. 2.11]{Duren} that $\kappa \in H^\infty$ since it has boundary values in $L^\infty(\partial K)$.

Let $\psi$ denote the inverse of $\phi$ and define $h(z) = \kappa(\psi(z))$.
Then $h \in H^\infty(\Omega)$ and
has boundary values $\overline{T(z)}$
almost everywhere on $\partial \Omega$.

Moreover, the function $\kappa$ is continuous up to the boundary $\partial K$ except possibly at the finitely many points that are mapped by $\phi$ to an infinite prime end of $\Omega$.
Indeed, being in $H^\infty$ the function $\kappa$ is representable by a Poisson integral of its boundary values $\ol{T\circ \phi}$, and as such it is continuous up to the boundary at all points of continuity of the boundary function $\ol{T\circ \phi}$. We have that $T$ is continuous by the smoothness assumption on $\partial \Omega$, and $\phi$ is continuous, except at the preimages of infinite prime ends, by the boundary behaviour of conformal mappings with Jordan boundary \cite[Ch. 2]{Pom}. This verifies the desired continuity of $\kappa$, and it then follows that $h$ is also continuous up to the boundary, and coincides with $\ol{T(z)}$ on $\partial \Omega$ everywhere, and not just almost everywhere.
\end{proof}

\begin{proof}[Proof of Theorem \ref{thm:main}]
As a candidate for the roof function
we take 
\be\label{eq:u}
u(z) = \Re \{ f(z) \} + C,
\ee
where
\begin{equation}
f(z) =  -i \int_{z_0}^{z} h(\zeta) d\zeta ,
\end{equation}
where $z_0 \in \Omega$ is fixed, and $C$ is an appropriate constant to be specified below.
From the continuity up to the boundary of $h$ we have $u \in C^1(\ol{\Omega})$, and we verify from the boundary values of $h$ that $\nabla u = \ol{f'(z)} = i\ol{h(z)}$ coincides with the inward-pointing unit normal vector.
Indeed, $h$ has boundary values $\ol{T(z)}$,
so that $\nabla u$ has boundary values $i T(z)$,
which is the unit normal vector.
Furthermore, we notice that $u$ is single-valued, 
and in fact, locally a constant on  $\partial \Omega$, 
since the integral $\int_{\gamma} h(\zeta) d\zeta = \int_{\gamma} \overline{T(\zeta)} d\zeta = \int_{\gamma} \overline{T(\zeta)} T(\zeta) ds $
is purely real for each subarc $\gamma$ of $\partial \Omega$.
The constant $C$ is chosen to ensure non-negativity of the piecewise-constant boundary values of $u$.

The following growth estimate will be key to showing the positivity of $u$.

\begin{lemma}\label{lemma:growth}
The function $u$ defined in \eqref{eq:u} satisfies the growth estimate
\begin{equation}\label{eq:linear}
 u(z) = O(|z|), \quad \text{as } z \rightarrow \infty.
\end{equation}
\end{lemma}
\begin{proof}[Proof of Lemma \ref{lemma:growth}]

Since $h \in H^\infty(\Omega)$
we have $|\nabla u| = O(1)$
as $z \rightarrow \infty$.

Consider the line segment running from the origin to $z$, and let $\ell$ be the connected component of this line segment that contains $z$.  Hence, $\ell$ is a line segment joining $z^*$ and $z$, where $z^*$ is either a point on $\partial \Omega$ or $z^*=0$.
Then we express
$u(z)$ as an integral
\begin{equation}\label{eq:ell}
u(z) = 
\int_{\ell} \langle \nabla u, r \rangle |dz| + u(z^*),
\end{equation}
where $\langle \nabla u, r \rangle$ denotes the inner product of $\nabla u$ with the unit vector $r$ in the direction of $\ell$.
This gives the desired estimate
\begin{equation}
|u(z)| \leq 
|z - z^*| |\nabla u| + |u(z^*)| = O(|z|),
\end{equation}
since $|\nabla u| = O(1)$
and $u(z^*)$ is either $u(0)$ or one of the finitely many Dirichlet boundary values.
\end{proof}

Finally, we show positivity of $u$ throughout $\Omega$ which we state as a lemma.  The proof of the theorem will be complete once we prove the lemma.

\begin{lemma}\label{lemma:pos}
The function $u$ defined in \eqref{eq:u} satisfies $u>0$ in $\Omega$.
\end{lemma}

\begin{proof}[Proof of Lemma \ref{lemma:pos}]
Following an idea from the proof of \cite[Thm. II]{Kjellberg} we prove the lemma by utilizing a method from the proof of the Denjoy-Carleman-Ahlfors theorem on asymptotic values of entire functions
(we will follow the presentation from \cite[Ch. 10, Thm. 5.4]{Evgrafov}).
Suppose $u(z_0) < 0$ for some $z_0 \in \Omega$.
Let $R_1$ denote the connected component containing $z_0$ of the set where $u<0$.
Notice that $R_1$ is unbounded 
(otherwise $u<0$ in a bounded region where it has zero boundary values and this violates the maximum principle).

We note that in the case when $\partial \Omega$ is compact, $\Omega$ is a disk.  This follows from \cite[Remark 6.1]{Gust} as noticed in \cite[Proof of Theorem 3.2]{KLT}. Then $\Omega$ admits a roof function (for instance, $u(z) = \log |z|$ when $\Omega$ is the unit disk). Hence, we may assume that at least one connected component of $\partial \Omega$ is unbounded.

Among the unbounded components of $\partial \Omega$, choose $L$ to be one for which the boundary value $m$ of $u$ along $L$ is maximal (recall there are finitely many boundary components and $u$ is constant along each of them). 
Let $S_m$ denote the set of points for which $u>m$.
Then $S_m$ is contained in $\Omega \setminus \ol{R_1}$,
and $S_m$ contains points near each point on $L$
by positivity of the inward normal derivative of $u$.
Among the connected components of $S_m$, let $\Omega_m$ denote the one that has $L \subset \partial \Omega_m$.
Let $\gamma$ denote a Jordan arc through $\Omega \setminus \ol{R_1}$ from a point on $L$ to a point on $\partial R_1$ and consider $\Omega_m \setminus \gamma$
which consists of two regions $\Omega_a$ and $\Omega_b$.

{\bf Claim.} $u \rightarrow \infty$ along a path to infinity in each of the regions $\Omega_a$ and $\Omega_b$.  

Indeed, suppose for the sake of contradiction that $u$ is bounded in one of these domains, say $\Omega_a$. Let $\phi_a : K_a \rightarrow \Omega_a$ be a conformal mapping from a circular domain $K_a$ with the same (finite) connectivity as $\Omega_a$.  Then $u \circ \phi_a$ is the solution to the Dirichlet problem in $K_a$ with continuous boundary values and constant $=m$ boundary values along the arc $\alpha := \phi_a^{-1}(\partial \Omega_a \setminus \gamma)$. Let $\alpha_1$ denote the open subarc $\phi_a^{-1}(L \cap \partial \Omega_a)$ of $\alpha$, and let $w_0$ denote the endpoint of $\alpha_1$ for which $\phi_a(w)$ is unbounded as $w \rightarrow w_0$ (recall that $L \cap \partial \Omega_a$ is unbounded).  By the reflection principle, $u \circ \phi_a$ extends to be harmonic in a neighborhood of $w_0$. This implies that $\nabla ( u \circ \phi)(w)$ is the conjugate of an analytic function in this same neighborhood.  (Here we are viewing the gradient $\nabla = \partial_x + i \partial_y$ in complex form.)  In particular, $\nabla ( u \circ \phi)(w)$ approaches a finite constant $c_0$ as $w \rightarrow w_0$. 
 We have $\nabla (u \circ \phi_a) = \nabla u (\phi_a) \cdot \overline{\phi_a'}$ in $K_a$ (again viewing the gradient vectors in complex form), and $\nabla u (\phi_a)$ is continuous up to the boundary at each point on the open arc $\alpha_1$ with  $|\nabla u(\phi_a)|=1$ along $\alpha_1$.  This implies that $\phi_a' = \overline{\nabla (u \circ \phi_a) / \nabla u (\phi_a)}$ is continuous up to the boundary at each point on $\alpha_1$ and that $|\phi_a'(w)| \rightarrow |c_0|$ as $w \rightarrow w_0$ along $\alpha_1$.  In particular, $|\phi_a(w)|$ is bounded (by a constant independent of $w$) for points $w$ near $w_0$ on the arc $\alpha_1$, a contradiction.  This proves the Claim.


Choosing $M_a$ to be the maximum of $u$ along $\partial \Omega_a$ and $M_b$ the maximum of $u$ along $\partial \Omega_b$, it follows from the Claim that the regions $R_2 := \{ z \in \Omega_a : u(z) > M_a\}$
and $R_3 := \{ z \in \Omega_b : u(z) > M_b \}$
are each nonempty.

We thus have three disjoint regions $R_1$, $R_2$, $R_3$ 
each unbounded, with $u(z) < 0$ in $R_1$, $u(z)>M_a$ in $R_2$, and $u(z)>M_b$ in $ R_3$.

For $k=1,2,3$ we define
$\theta_k(t)$
to be the length of 
$$\{z\in R_k : |z|=t \} .$$

Let $M(r) := \max_{|z| \leq r} |u(z)|$,
and for $k=1,2,3$ let $M_k(r):=\max_{|z| \leq r, z \in R_k} |u(z)|.$

By the Phragmen-Lindelof principle \cite[Thm. 6.1]{Evgrafov} (the theorem is stated for an analytic function $f$ but only relies on the subharmonicity of $\log|f|$ and thus can easily be adapted replacing $\log|f|$ with $|u|$ which is harmonic in each of the regions $R_k$),
we have for $k=1,2,3$
$$\log M_k(r) \geq \pi \int_1^{r} \frac{1}{\theta_k(t)} dt +C ,$$
where $C$ is a constant depending on $u$ but independent of $r$.
We also have for $k=1,2,3$
$\log M(r) \geq \log M_k(r)$,
and hence
\begin{equation}\label{eq:3logM}
3 \log M(r) \geq \pi \int_1^{r} \sum_{k=1}^3 \frac{1}{\theta_k(t)} dt +3C.
\end{equation}

Since
$$\sum_{k=1}^3 \theta_k(t) \leq 2\pi t ,$$
we have (using the Cauchy-Schwarz inequality)
\begin{align*}
  2\pi t \sum_{k=1}^3 \frac{1}{\theta_k} &\geq \sum_{k=1}^3 \theta_k \sum_{k=1}^3 \frac{1}{\theta_k} \\ 
  &\geq  \left( \sum_{k=1}^3 \sqrt{\theta_k} \sqrt{\frac{1}{\theta_k}}\right)^2 = 3^2,
\end{align*}
which implies
$$ \pi \sum_{k=1}^3 \frac{1}{\theta_k(t)} \geq  \frac{9}{2t}. $$
Integrating gives
$$  \pi \int_1^r \sum_{k=1}^3 \frac{1}{\theta_k(t)}dt \geq \int_1^r \frac{9}{2t} dt = (9/2) \log r, $$
and combining this with \eqref{eq:3logM}
we obtain
$$ \log M(r) \geq (3/2) \log r + C.$$
This contradicts \eqref{eq:linear}
which states that $|u(z)| = O(|z|)$
as $z\rightarrow \infty$,
and we conclude that $u>0$ throughout $\Omega$.  This concludes the proof of the lemma.
\end{proof}

Applying Lemma \ref{lemma:pos} completes the proof of the theorem.
\end{proof}

\begin{remark}
The anonymous referee has kindly pointed out that part of the proof of Lemma \ref{lemma:pos} (namely, the part after the proof of the Claim, where a contradiction is derived from the existence of the three regions $R_1,R_2,R_3$) can be simplified using the following result of M. Heins (while taking $n=3$ and for $j=1,2,3$ defining $u_j$ to equal $|u|$ minus its constant boundary values in $R_j$ and zero elsewhere).
\begin{thm}\cite[Thm. 5.1]{Heins}
Let $u_1,u_2,...,u_n$ denote $n \geq 2$ non-constant, non-negative subharmonic functions in the plane which satisfy $\min(u_j,u_k)=0$ for $j \neq k$. Then
$$ \liminf_{r \rightarrow \infty} r^{-n/2} \left( \sum_{k=1}^n \int_0^{2\pi} u_k^2(r e^{i \theta}) d\theta \right)^{1/2} > 0 .$$
\end{thm}
\end{remark}


\section{Concluding Remarks}\label{sec:cor}

Theorem \ref{thm:main} shows that
in the definition of the roof function
(at least for domains with finitely many boundary components) the condition of positivity can be replaced by a growth condition $u(z) = O(|z|)$
while only imposing positivity on the boundary values,
and the positivity of $u$ follows automatically.
Indeed,
the boundary condition along with the growth condition
imply the arclength NQD condition \eqref{eq:ALNQD} as explained in Section \ref{sec:q}.
Then Theorem \ref{thm:main},
with some attention to the details of its proof, implies positivity of $u$.

Theorem \ref{thm:main} allows immediate application of several results from \cite{EL}
to the classification  of arclength NQDs that we shall summarize below.

Assume that $\partial \Omega$ has finitely many connected components.
Then Theorem \ref{thm:main} shows that $\Omega$ is a quasi-exceptional domain.
This implies \cite{EL} that the number of unbounded components of $\partial \Omega$ is either zero, one, or two, and we have the following partial classification (see \cite{EL}).

\begin{itemize}
 \item $\partial \Omega$ compact $\implies$ $\Omega$ is the exterior of a disk
 \item exactly one component of $\partial \Omega$ is unbounded $\implies$ $\Omega$ is a halfplane
 \item two components of $\partial \Omega$ are unbounded and $\Omega$
 is simply-connected $\implies$ $\Omega$ is the Hauswirth-Helein-Pacard example \cite{HHP}
\end{itemize}

This leaves open the case that two components of $\partial \Omega$ are unbounded and $\Omega$ is multiply-connected. 
This category appears to be the most interesting.
Doubly-connected examples were constructed using elliptic functions in \cite{EL}, where it is conjectured that there exist examples with every connectivity.
See Figure \ref{fig:cases}.

\begin{figure}[ht]
    \centering
\hspace{0.0001in}
    {
\setlength{\fboxsep}{0pt}%
\setlength{\fboxrule}{1pt}%
\fbox{\includegraphics[width=0.4\linewidth]{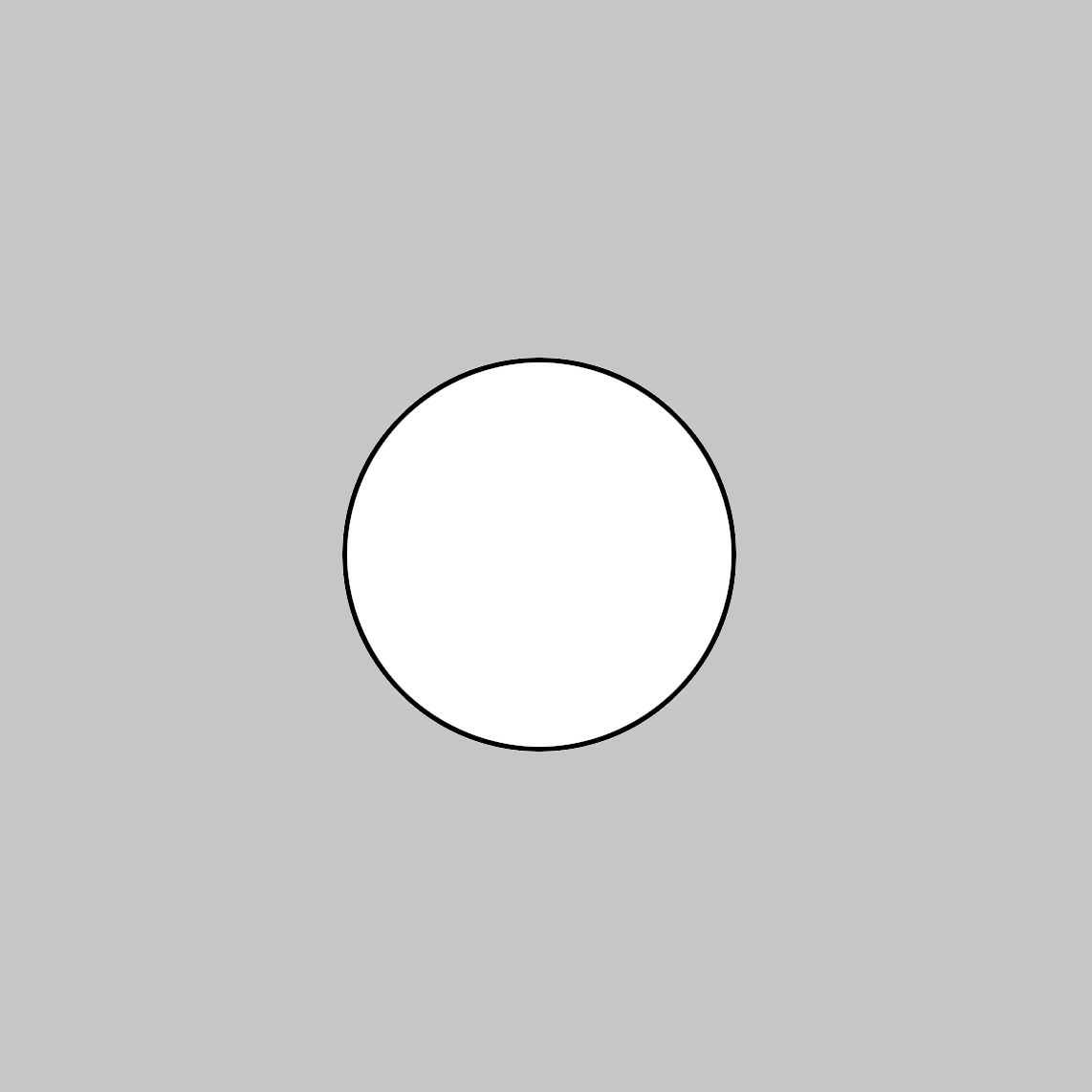}}
}
{
\setlength{\fboxsep}{0pt}%
\setlength{\fboxrule}{1pt}%
\fbox{\includegraphics[width=0.4\linewidth]{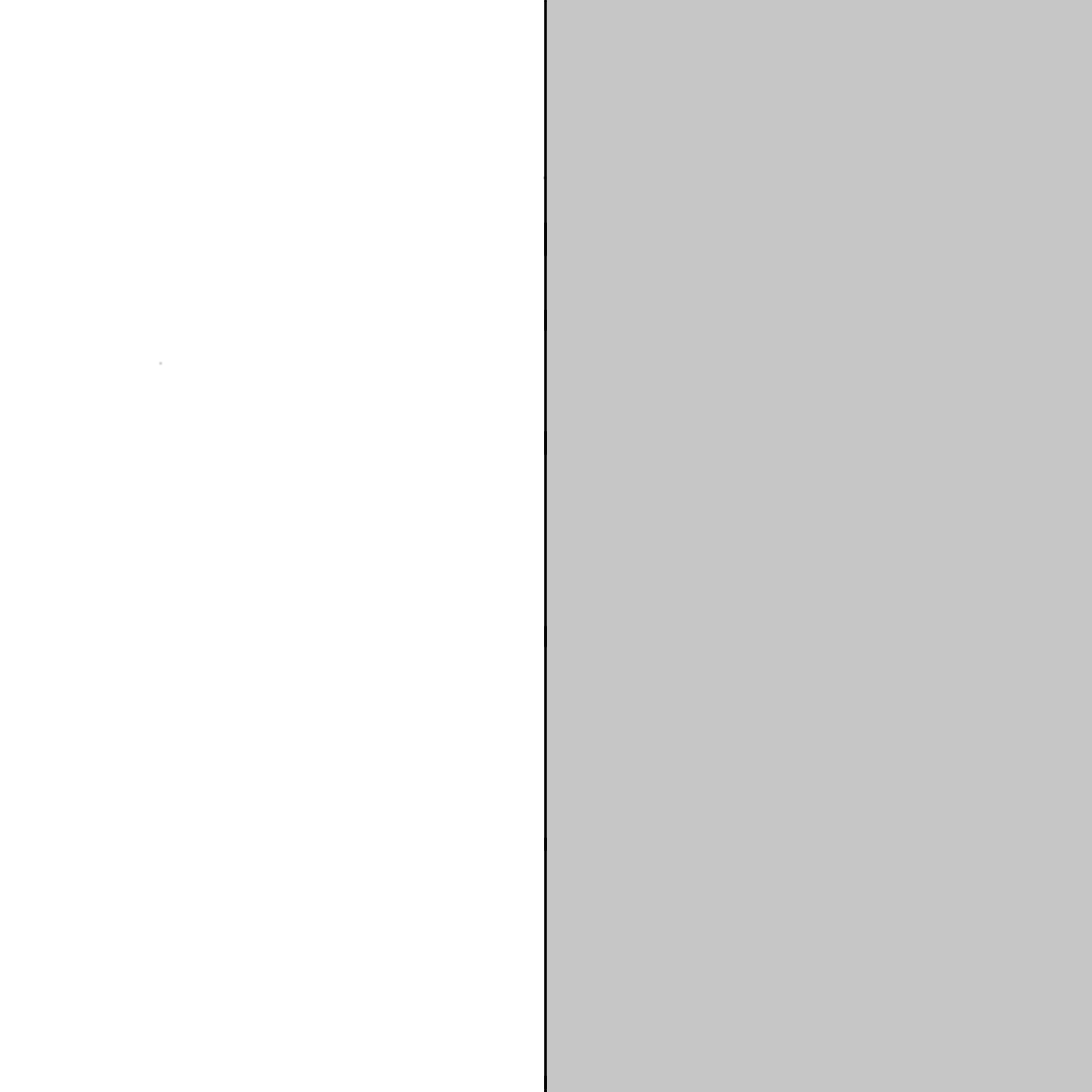}}
} 
\\
\vspace{0.11in}
\hspace{0.04in}
{
\setlength{\fboxsep}{0pt}%
\setlength{\fboxrule}{1pt}%
\fbox{\includegraphics[width=0.4\linewidth]{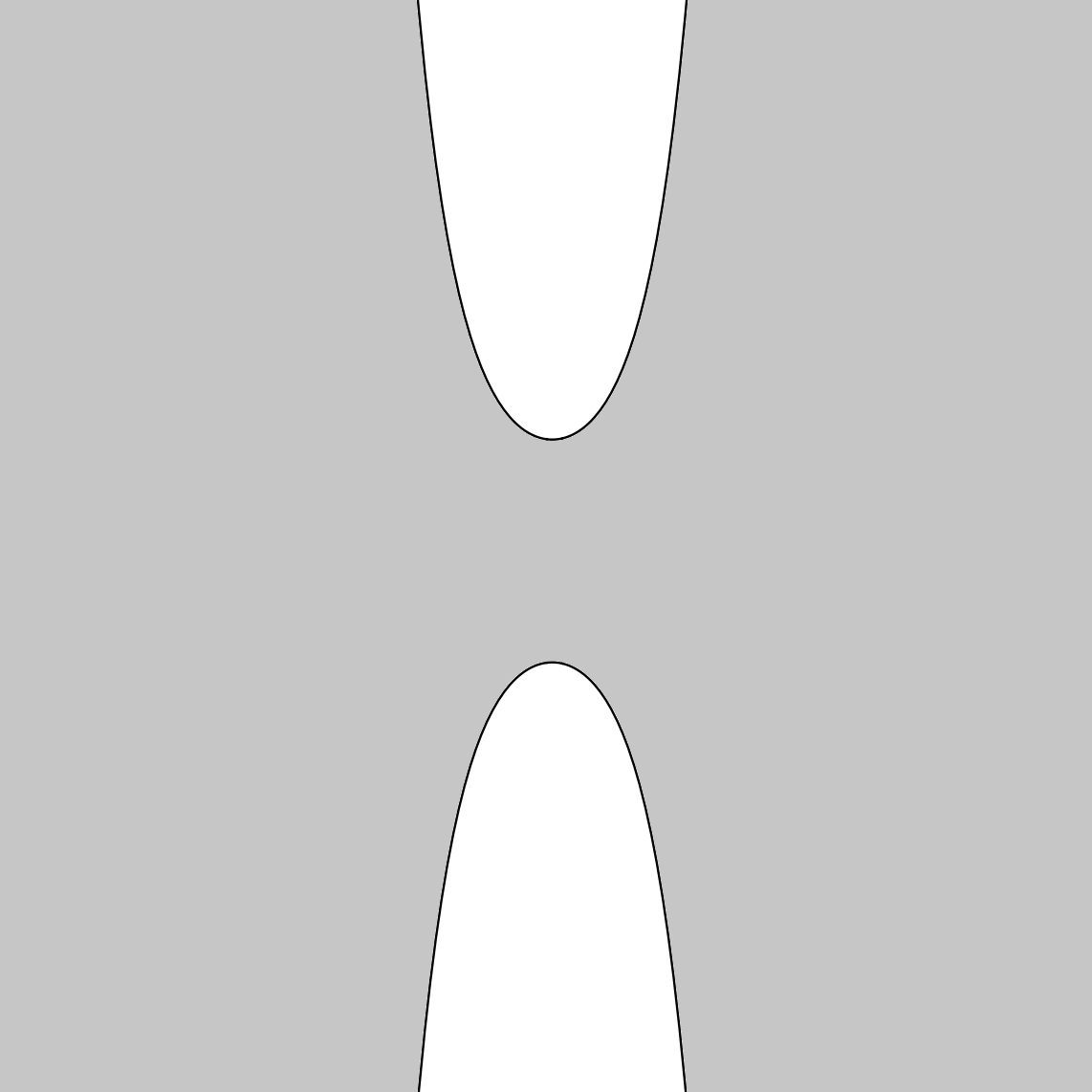}}
}
{
\setlength{\fboxsep}{0pt}%
\setlength{\fboxrule}{1pt}%
\fbox{\includegraphics[width=0.4\linewidth]{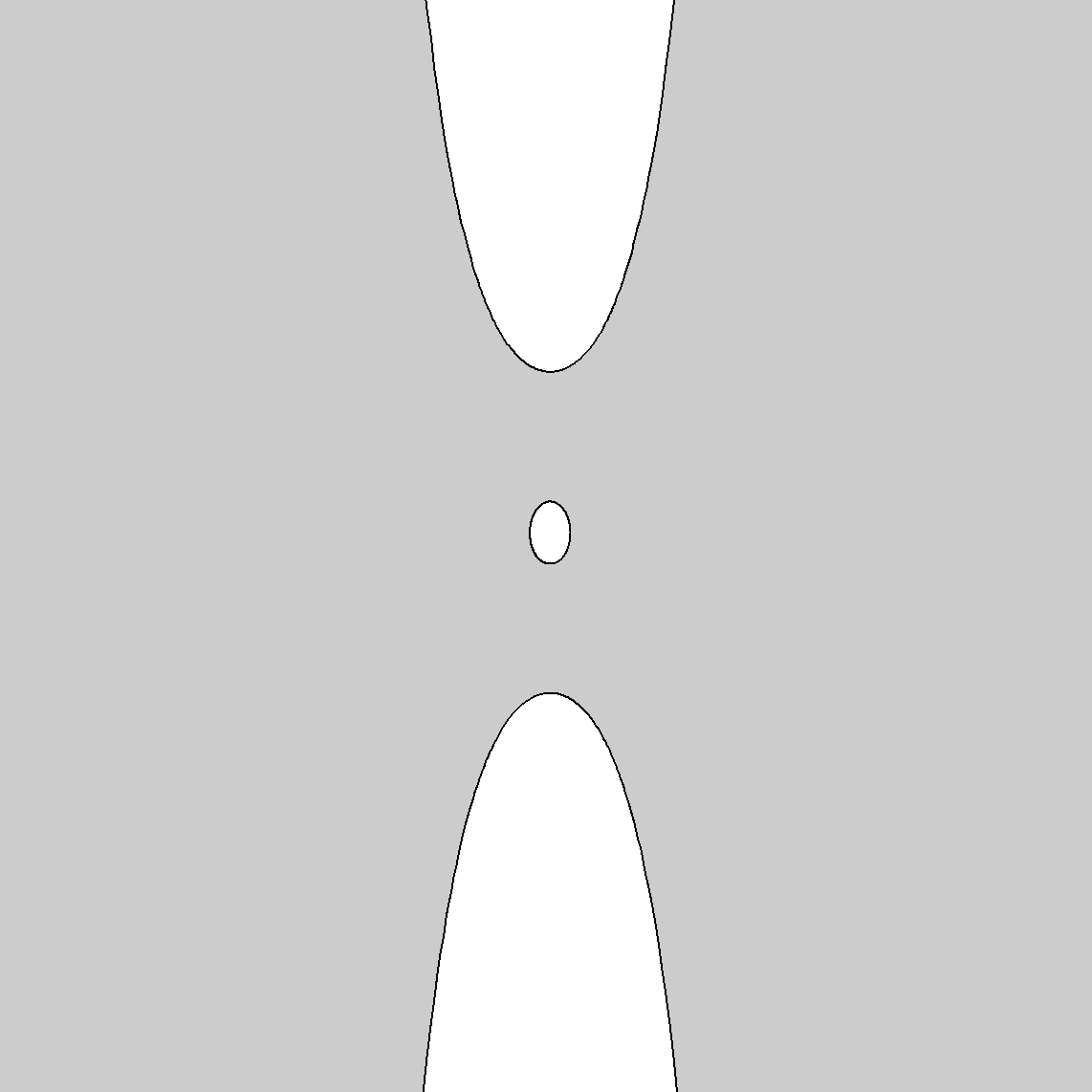}}
}
    \caption{Illustration summarizing progress on classification of arclength NQD with finitely many boundary components.  We have that the number of unbounded components of $\partial \Omega$ is either zero, one or two.  In the case $\partial \Omega$ is compact, $\Omega$ is the exterior of a disk (top left).  In the case exactly one component of $\partial \Omega$ is unbounded, $\Omega$ is a halfplane (top right).  In the case two components of $\partial \Omega$ are unbounded and $\Omega$ is simply-connected, $\Omega$ is (up to translation/rotation) $ \left\{ x+iy \in \CC : - \frac{\pi}{2}-\cosh x < y < \frac{\pi}{2} + \cosh x \right\}$ (bottom left).  The case when two components of $\partial \Omega$ are unbounded and $\Omega$ is multiply-connected remains open.  A doubly-connected example is shown (bottom right).}
    \label{fig:cases}
\end{figure}

\bibliographystyle{abbrv}
\bibliography{NQD}

\end{document}